\documentclass[reqno]{amsart}

\usepackage[a4paper, left=1.4in,right=1.4in]{geometry}

\usepackage{mathrsfs}
\usepackage{amsfonts}
\usepackage{amsmath}
\usepackage{amssymb}
\usepackage{mathtools}
\usepackage{float}
\usepackage{xcolor}
\usepackage[inline]{enumitem}
\usepackage{tikz}
\usepackage{hhline}
\usepackage{booktabs}

\DeclareMathOperator{\FIM}{FIM}
\DeclareMathOperator{\FG}{FG}

\newtheorem{theorem}{Theorem} 
\newtheorem*{theorem*}{Theorem} 

\newtheorem{lemma}[theorem]{Lemma}     
\newtheorem{corollary}[theorem]{Corollary}

\newtheorem{proposition}[theorem]{Proposition}

\numberwithin{theorem}{section}

\theoremstyle{definition}

\newtheorem{question}[theorem]{Question}
\newtheorem*{question*}{Question}
\newtheorem{example}[theorem]{Example}
\newtheorem{remark}[theorem]{Remark}

\DeclareMathOperator{\N}{\mathbb{N}}

\begin{document}

\title[The growth of free inverse monoids]{The growth of free inverse monoids}

\author[Kambites]{Mark Kambites}
\address{Department of Mathematics, University of Manchester, Manchester M13 9PL, UK}
\email{mark.kambites@manchester.ac.uk}

\author[Nyberg-Brodda]{Carl-Fredrik Nyberg-Brodda}
\address{School of Mathematics, Korea Institute for Advanced Study (KIAS), Seoul 02455, Republic of Korea} 
\email{cfnb@kias.re.kr}

\author[Szak\'acs]{N\'ora Szak\'acs}
\address{Department of Mathematics, University of Manchester, Manchester M13 9PL, UK}
\email{nora.szakacs@manchester.ac.uk}

\author[Webb]{Richard Webb}
\address{Department of Mathematics, University of Manchester, Manchester M13 9PL, UK}
\email{richard.webb@manchester.ac.uk}

\thanks{The second author is supported by the Mid-Career Researcher Program (RS-2023-00278510) through the National Research Foundation funded by the government of Korea, and by the KIAS Individual Grant (MG094701) at Korea Institute for Advanced Study. Part of the research was carried out while the second author was supported by the Dame Kathleen Ollerenshaw Trust.}

\date{\today}


\begin{abstract}
We compute the rate of exponential growth of the free inverse monoid of rank $r$ (and hence an upper bound on the corresponding rate for all $r$-generated inverse monoids and semigroups). This turns out to be an algebraic number strictly between the obvious bounds of $2r-1$ and $2r$, tending to $2r$ as the rank tends to infinity. We also find an explicit expression for the exponential growth rate of the number of idempotents, and prove that this tends to $\sqrt{e(2k-1)}$ as $k \to \infty$.
\end{abstract}

\maketitle

\section{Introduction}

\noindent The asymptotic growth rate of finitely generated groups is by now a classical topic in combinatorial and geometric group theory. Its origin lies in independent work by Schwartz \cite{Schwartz1955} and Milnor \cite{Milnor1968} on the growth of the fundamental groups of compact Riemannian manifolds. Finitely generated groups with polynomial growth were eventually characterized by Wolf \cite{Wolf1968} and Gromov \cite{Gromov1981} as being precisely the virtually nilpotent groups. There is a wealth of literature on growth rates of groups, and we refer the reader to \cite{Grigorchuk2014} for an introduction to the subject and its history. Studying the growth of semigroups (and algebras in general) is a very natural problem arising purely from combinatorial considerations, with classic results by Bergman \cite{Bergman1978} and many recent breakthrough results in the area, e.g.\ by Bell \& Zelmanov \cite{Bell2021}.

In the past few years, there has been a flurry of activity in combinatorial and geometric \textit{inverse} semigroup theory. A number of exciting and fundamental results have been proved, including a recent proof of the undecidability of the word problem for one-relation inverse monoids \cite{GrayInv}, as well as work on related algorithmic problems \cite{CFIJAC, GrayNik} and geometric questions \cite{Nora2022}, but many mysteries remain. In particular, the study of the growth of inverse semigroups has not yet played a significant role in this development (cf.\ \S\ref{Subsec:growth-of-inverse-monoids}). Some results concerning e.g.\ (non-)automaticity of free inverse monoids have already been proved using growth \cite{Cutting2001}, and it is therefore conceivable that \textit{algebraic} information can also be extracted from the growth of inverse monoids. 

The first fundamental question that appears in this line of thought is the following: what is the growth of \textit{free} inverse monoids? These are the free objects in the variety of inverse monoids, and can be constructed in several different ways. A direct approach is to construct the free inverse monoid $\FIM_r$ of rank $r$ as a $2r$-generated monoid, subject to some defining relations. These (known) relations all hold in free groups, giving two pieces of information on the exponential growth rate of $\FIM_r$ when $r >1$: it must lie between that of the free group of rank $r$ and that of the free monoid of rank $2r$, i.e.\ between $2r-1$ and $2r$. It is, however, far from obvious where in this interval it should lie.

In this article, we solve this and a closely related problem. Free inverse monoids have an elegant geometric representation, due in its modern form to Munn \cite{Munn1974} but with key ideas going back to work of Scheiblich \cite{Scheiblich1973}, and we exploit this representation to compute the growth rate of the free inverse monoid $\FIM_r$ of rank $r>1$. In \S\ref{Sec:sphere-counting}, we give a closed formula for the size of the spheres in a free inverse monoid. We then analyse this formula asymptotically in \S\ref{Sec:Idempotent-growth} and \S\ref{Sec:General-exponential-growth-rate}. Specifically, in \S\ref{Sec:Idempotent-growth} we use the formula to prove our first result (Theorem~\ref{Thm:growth-of-idempotents}) determining the exponential growth rate of the number of idempotents in $\FIM_r$. This turns out to be exactly
\[
\left(\frac{2r-1}{2r-2}\right)^{r-1} \sqrt{2r-1}.
\]
which approaches $\sqrt{e(2r-1)}$ as the rank $r$ tends to infinity.

\begin{table}
\begin{tabular}{|l|*{6}{c|}}
\hline
Rank $r = $ & 2 & 3 & 4 & 5 & 6 & 7 \\
\hline
\specialrule{.05em}{0em}{.05em}
\hline
\textbf{Growth rate of} $\FIM_r$ $\approx$ & 3.636 & 5.759 & 7.819 & 9.855 & 11.878 & 13.896 \\
\hline
\end{tabular}
\vspace{5pt}
\caption{The exponential growth rate of the free inverse monoid for small ranks $r$, being the largest real root of \eqref{Eq:main-polynomial} for $p=2r-1$.}
\label{Table:growth-table}
\end{table}

In \S5, we prove our second main result (Theorem~\ref{expg}), determining the exponential growth rate of free inverse monoids. This rate turns out to be an algebraic number, namely the largest real root of the polynomial 
\begin{equation}\label{Eq:main-polynomial}
p^py^{p-2} - (py-1)^{p-1} = 0
\end{equation}
where $p=2r-1$. For $r \leq 7$, these rates are presented in Table~\ref{Table:growth-table}. As $r \to \infty$, the rate tends to $2r$, albeit rather slowly. Our computation does not give any real insight into why this particular polynomial arises. It seems nearly always to be irreducible over $\mathbb{Q}$ (see Remark~\ref{Rem:remark-irr}) and therefore (up to scaling to make it monic) is the minimal polynomial of the growth rate, making it an invariant of free inverse monoids. It would be interesting to know whether there is some more conceptual reason why this polynomial is associated to the free inverse monoid.

\section{Preliminaries}\label{Sec:Preliminaries}

\noindent In this section, we will briefly recall the definition and key properties of free inverse monoids and of growth. We assume the reader is familiar with the rudiments of semigroup theory; we refer the reader to the book by Lawson \cite{Lawson1998} for a detailed introduction to inverse monoids, and particularly \cite[Chapter~6]{Lawson1998} for an introduction to free inverse monoids and Munn trees. Throughout this article, for a finite set $X$ we let $X^\ast$ be the free monoid of all (finite) words over $X$, and we denote by $\FG(X)$ the free group with basis $X$.

\subsection{Growth} Let $M$ be a (not necessarily inverse) monoid $M$ generated (as a monoid) by a finite set $X$. Recall that the \textit{length} of an element $m \in M$ is
the smallest $k$ such that $M$ can be written as a word of length $k$ over the generating set $X$. For $K \in \mathbb{N}$ the \textit{sphere of radius $K$} is the set of elements in $M$ of length exactly $K$, and it is denoted by $S(K)$. The \textit{(spherical) growth function} of $M$ (with respect to $X$) is the function $\gamma \colon \mathbb{N} \to \mathbb{N}$ defined by $K \mapsto |S(K)|$. The \textit{(exponential) growth rate} of $M$ is defined to be the asymptotic exponential growth rate of the function $\gamma$, that is,
\begin{equation}\label{Eq:limsup-definition-growth}
\limsup_{n \to \infty} \gamma(n)^{\frac{1}{n}}.
\end{equation}
Since the growth rate is submultiplicative, i.e.\ $\gamma(m+n) \leq \gamma(m)\gamma(n)$ for all $m, n \in \N$, the limit \eqref{Eq:limsup-definition-growth} always exists. If the limit \eqref{Eq:limsup-definition-growth} is greater than $1$, then $M$ is said to have \textit{exponential growth}. If there are natural numbers $C, d \in \N$ such that $\gamma(n) \leq (Cn)^d$ for all $n \in \N$, then we say that $M$ has \textit{polynomial growth} (of degree $d$, where $d$ is the smallest such). All growth in this article, with the sole exception of the monogenic free inverse monoid treated in \S\ref{Subsec:monogenic} (see Proposition~\ref{Prop:FIM1-growth}), will be exponential. 

\begin{remark}
A natural alternative definition of growth can be given in terms of \textit{balls} rather than spheres; it is well-known and easy to see that this leads to the same value for the exponential growth rate (and increases the degree by one for polynomial growth). We choose to work with spheres because it happens to simplify our calculations.
\end{remark}

\subsection{Free inverse monoids}

Inverse monoids form a variety of algebras, and thus free inverse monoids exist; see \cite[Chapter~6, \S1]{Lawson1998}. We denote the free inverse monoid on $X$ by $\FIM(X)$, and $|X|$ is called the \textit{rank} of $\FIM(X)$. Then $\FIM(X)$ is generated as a monoid by $X \cup X^{-1}$, where $X^{-1}$ is a set in bijective correspondence with $X$ and such that $X \cap X^{-1} = \varnothing$, and we will study its growth with respect to this generating set. If $|X|= k$, we shall sometimes write $\FIM_k$ for $\FIM(X)$, and call this the free inverse monoid of rank $k$. Free inverse monoids have many peculiar properties differentiating them from free groups. For example, Schein \cite{Schein1975} proved that $\FIM_k$ is not finitely presented as a monoid, even when $k=1$ (although it is finitely generated). Indeed, $\FIM_1$ does not have the homological finiteness property $\operatorname{FP}_2$ (and hence neither does any $\FIM_k$ for $k \geq 1$), a strictly weaker property than finite presentability \cite{Gray2021}. Neither does the analogue of the Nielsen--Schreier Theorem hold for $\FIM_k$: there exist non-free inverse submonoids of $\FIM_k$, see \cite{Reilly1972}. Even solving the word problem in $\FIM_k$ is a non-trivial task, see \cite{Scheiblich1973, Munn1974}; this is normally done via \textit{Munn trees}, which we now define.

\subsection{Munn trees}\label{Subsec:munn-tree-intro}

The model of free inverse monoids most commonly used in semigroup theory is due to Munn \cite{Munn1974}, which we briefly recall. This represents elements of $\FIM(X)$ as birooted trees called Munn trees, and also provides a compelling link between $\FIM(X)$ and the Cayley graph $\Gamma_X$ the free group $\FG(X)$. Elements of $\FIM(X)$ are uniquely represented by the pairs $(T,g)$, where $T$ is a connected subgraph (and thus a subtree) of $\Gamma_X$ containing the vertex $1$, and $g$ is a vertex of $T$. The product of two Munn trees $(T_1,g_1)$ and $(T_2,g_2)$ is $(T_1 \cup g_1\cdot T_2, g_1g_2)$, where $g_1\cdot T_2$ denotes the image of $T_2$ under the left action of $g_1$. The Munn tree corresponding to a generator $x \in X \cup X^{-1}$ is the pair $(\langle x \rangle, x)$ where $\langle x \rangle$ is the subgraph of $\Gamma_X$ spanned by the unique edge that $x$ labels from $1$. It follows that the Munn tree represented by a product $w \in (X \cup X^{-1})^\ast$ of generators is the pair $(\langle w \rangle, w_{\FG(X)})$ where $\langle w \rangle$ is the subgraph of $\Gamma_X$ spanned by the unique path that $w$ labels from $1$, and $w_{\FG(X)}$ is the element of $\FG(X)$ represented by $w$ (i.e.\ its freely reduced form), which is also the terminal vertex of this path. Clearly, a Munn tree $(T,g)$ is idempotent if and only if $g=1$.

\subsection{Growth of inverse monoids}\label{Subsec:growth-of-inverse-monoids}

The growth rate of the free inverse monoid of rank $r$ is also automatically an upper bound on the growth rate of any $r$-generated inverse monoid or semigroup. The growth of different classes of inverse monoids has been studied, and although we cannot hope to survey this subject in its entirety here, there are many tantalizing open problems in this area. For example, classifying precisely what the growth of one-relation inverse monoids can be remains an open problem, cf.\ \cite[p.~235]{Shneerson2015}. Furthermore, Schneerson \& Easdown \cite{Shneerson2011} have studied the growth of certain Rees quotients of free inverse monoids; in \cite{Easdown2013} the same authors give sufficient conditions for a one-relation inverse monoid, i.e.\ the quotient of $\FIM(X)$ by some relation $w_1 = w_2$, to have exponential growth. One such sufficient condition is that $|X|=2$ and that the Munn trees of both $w_1$ and $w_2$ contain more than one edge. One natural direction for future research would be to investigate the precise exponential growth rates in these cases: in particular, whether they are always algebraic numbers, as for $\FIM(X)$ (as we shall prove in Theorem~\ref{expg}). We also remark that methods arising from growth were used by Cutting \& Solomon \cite{Cutting2001} to show that $\FIM_1$ is not automatic (and hence neither is $\FIM_r$ for $r \geq 1$), and indeed does not admit a regular language of normal forms. 

\section{The size of spheres}\label{Sec:sphere-counting}

\noindent In order to calculate the growth rate of free inverse monoids, the first step is to determine the size of spheres. This is our goal for the section. We fix some notation and terminology which will be used throughout this section. Let $X$ be a fixed finite set. The unique geodesic in $T$ from $1$ to $g$ is called the \textit{trunk} of the Munn tree $(T,g)$. The edges of $T$ which are not on the trunk will be called \textit{branch edges}. The Munn trees with no trunk edges are precisely the idempotents of $\FIM(X)$.

To count the number of Munn trees with a given length, we begin by counting the number of Munn trees with a given trunk. To do this, we will relate this problem to the \textit{Fuss-Catalan numbers} (also sometimes known as \textit{Raney numbers}). We first recall some results regarding these numbers.

\subsection{Tree diagrams of branching $(p,q)$}

Consider the infinite rooted (undirected) tree $T_{p,q}$ where the root has $q$ children and all other vertices have $p$ children. The finite subtrees containing the root are called \emph{tree diagrams with branching $(p,q)$} -- these are rooted trees where the root has at most $q$ children, and all other vertices have at most $p$ children, but considered as subgraphs of the infinite tree rather than up to isomorphism; for example, the two graphs in Figure~\ref{Fig:two-branching-trees} are two distinct tree diagrams with branching $(1,2)$.
\begin{figure}[h]
	\begin{tikzpicture}
		\draw[thick] (0,0) node[fill=black, circle, scale=0.5]{} -- (1,-1) node[fill=black, circle, scale=0.5]{};
		\draw[thick] (0,0)  -- (-1,-1) node[fill=black, circle, scale=0.5]{};
		\draw[thick] (-1,-1) node[fill=black, circle, scale=0.5]{} -- (-1,-2.4) node[fill=black, circle, scale=0.5]{};
		\draw[thick] (3,0) node[fill=black, circle, scale=0.5]{} -- (4,-1) node[fill=black, circle, scale=0.5]{};
		\draw[thick] (3,0)  -- (2,-1) node[fill=black, circle, scale=0.5]{};
		\draw[thick] (4,-1) node[fill=black, circle, scale=0.5]{} -- (4,-2.4) node[fill=black, circle, scale=0.5]{};
	\end{tikzpicture}
	\caption{Two distinct tree diagrams with branching $(1,2)$.}
	\label{Fig:two-branching-trees}
\end{figure}
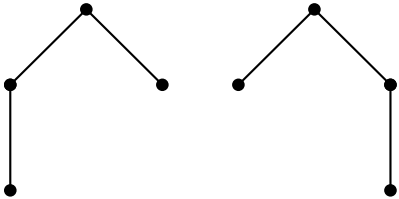
For $p=q$, we simply say \emph{tree diagrams with branching $p$}. Of particular importance to us are the \emph{$p$-ary tree diagrams}, which are those tree diagrams with branching $p$ where every internal node has exactly $p$ children. 

The \textit{Fuss-Catalan numbers} $R_{p,q}(k)$ are the numbers defined (cf.\ \cite[Theorem 2.5]{BD}) by the formula
\[
R_{p,q}(k)=\frac{q}{kp+q} \binom{kp+q}{k}=\frac{q}{k} \binom {kp+q-1}{k-1}.
\]
Fuss-Catalan numbers are important for us because they count tree diagrams of branching $(p,q)$; we believe this fact to be folklore, but include a proof as we have been unable to find a reference.

\begin{proposition}
	The number of tree diagrams with branching $(p,q)$ on $k$ edges is the Fuss-Catalan number $R_{p,q}(k)$.
\end{proposition}	

\begin{proof}
	Let $D_{p,q}(k)$ denote the number of tree diagrams with branching $(p,q)$ on $k$ edges. We show that $D_{p,q}(k)$ satisfies the same recursion as $R_{p,q}(k)$. By \cite[Lemma 2.4]{BD}, $R_{p,q}(k)$ satisfies the equation
\[	
R_{p,q}(k)=\sum_{\substack{i_1+\cdots+i_q=k \\ i_1, \dots, i_q \geq 0}} {}_pc_{i_1}\cdots {}_pc_{i_q},
	\]
where $_pc_i$ denotes the $i$th $p$-Catalan number, and where $p\geq 1$ and $i \geq 0$.
	
The $p$-Catalan numbers, also called the generalized Catalan numbers, have been widely studied. In particular (see e.g. \cite[Theorem~0.2]{HP}), the number $_pc_i$ is equal to the number of $p$-ary tree diagrams with $i$ internal nodes. We claim  that these are in bijection with the set of tree diagrams with branching $p$ on $i$ nodes, together with the empty set.
	Indeed, if $T$ is any tree diagram with branching $p$ on $i$ nodes, it can be turned into a $p$-ary tree with $i$ internal nodes by completing all $i$ nodes to have $p$ children, and conversely, deleting all leaves from a $p$-ary tree diagram with $i \geq 1$ internal nodes gives a tree diagrams with branching $p$ on $i$ nodes; see Figure~\ref{Fig:bijection-tree-diagram-ary}. The single $p$-ary tree diagram with $i=0$ internal nodes is the tree diagram consisting of a single vertex, and this corresponds to the empty set under this bijection.
	\begin{figure}[H]
		\begin{tikzpicture}
			\draw[thick] (-1,0) node[fill=black, circle, scale=0.5]{} -- (0,-1) node[fill=black, circle, scale=0.5]{};
			\draw[thick] (-1,0)  -- (-2,-1) node[fill=black, circle, scale=0.5]{};
			\draw[thick] (-2,-1) -- (-1.25,-1.75) node[fill=black, circle, scale=0.5]{};
			\draw[thick] (4.5,0) node[fill=black, circle, scale=0.5]{} -- (5.5,-1) node[fill=black, circle, scale=0.5]{};
			\draw[thick] (4.5,0)  -- (3.5,-1) node[fill=black, circle, scale=0.5]{};
			\draw[thick] (3.5,-1) -- (4.25,-1.75) node[fill=black, circle, scale=0.5]{};
			\draw[thick, dashed] (3.5,-1) -- (2.75,-1.75) node[draw=black, solid, circle, scale=0.5]{};
			\draw[thick, dashed] (4.25,-1.75) -- (4.25+0.5,-1.75-0.5) node[draw=black, solid, circle, scale=0.5]{};
			\draw[thick, dashed] (4.25,-1.75) -- (4.25-0.5,-1.75-0.5) node[draw=black, solid, circle, scale=0.5]{};
			\draw[thick, dashed] (5.5,-1) -- (4.75,-1.75) node[draw=black, solid, circle, scale=0.5]{};
			\draw[thick, dashed] (5.5,-1) -- (6.35,-1.75) node[draw=black, solid, circle, scale=0.5]{};
			\draw[->] (0.75,-0.5) to[bend left] (2.25,-0.5);
			\draw[->] (2.25,-1.5) to[bend left] (0.75,-1.5);
		\end{tikzpicture}
		\caption{The bijection taking a tree diagram with branching 2 on 4 nodes to a 2-ary tree diagram with 4 internal nodes}
		\label{Fig:bijection-tree-diagram-ary}
	\end{figure}
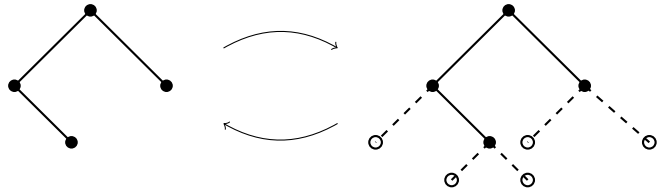
	
	Let us count the number $D_{p,q}(k)$ of tree diagrams with branching $(p,q)$ on $k$ edges. Let $T \subseteq T_{p,q}$ be such a tree diagram, and for all $1\leq j \leq q$, consider the subtree $T^j_{p,q}$ of $T_{p,q}$ induced on all descendants of the $j$th vertex of level one. Define $T_j=T^j_{p,q} \cap T$ -- this is either a tree diagram with branching $p$, or empty. Notice that $T$ is uniquely determined by the tuple $(T_1, \ldots, T_q)$. Every node of $T$ but the root is contained in exactly one of the subtrees $T_j$, which gives
	$V(T)=V(T_1)+\ldots+V(T_q)+1$, hence $k=V(T)-1=V(T_1)+\ldots+V(T_q)$. Thus $D_{p,q}(k)$ is equal to the number of $q$-tuples $(T_1, \ldots, T_q)$ where each $T_j$ is either a tree diagram with branching $p$, or empty, and where $V(T_1)+\ldots+V(T_q)=k$. The number of such tuples is exactly $\displaystyle\sum_{i_1+\cdots+i_q=k} {}_pc_{i_1}\cdots {}_pc_{i_q}$ as needed.
\end{proof}

We are now ready to use these combinatorial results to count Munn trees, and by extension also enumerate the spheres of a given radius in $\FIM(X)$. 

\subsection{Counting Munn trees}\label{Subsec:Counting-munn-trees}

Let $(T,g) \in \FIM(X)$ be a Munn tree. The length of $(T,g)$ can be expressed directly in terms of trunk edges and branch edges, as follows. 

\begin{lemma}\label{Lem:Munn-tree-length}
The length of a Munn tree $(T,g)$ with $t$ trunk edges and $k$ branch edges is equal to $t+2k$. 
\end{lemma}
\begin{proof}
Let $(T,g)$ be a Munn tree with $t$ trunk edges and $k$ branch edges. Then $(T,g)$ is represented by the product $w \in (X \cup X^{-1})^\ast$ if and only if $w$ labels a path from $1$ to $g$ in $T$ such that this path traverses all edges of $T$. By induction on word length, any word labelling a path from $1$ to $g$ must traverse each trunk edge an odd number of times, and each branch edge an even number of times. Hence, any word $w$ representing $(T,g)$ must have length at least $t + 2k$, i.e.\ $|w| \geq t+2k$. On the other hand, it is not difficult to construct a word $w$ of length $t+2k$ representing $(T,g)$: at any vertex,  depth-first traverse the subtree of branch edges attached to that vertex, then continue along the trunk and repeat. The label of the resulting path will then clearly have length $t+2k$, and the Munn tree of this word is exactly $(T,g)$. Hence the length of a shortest word representing $(T,g)$ is exactly $t+2k$, which is what was to be shown.
\end{proof}

Consider the set of all Munn trees with a given trunk. That is, fix $g \in \FG(X)$ and consider the set $\{ (T,g) : T \hbox{ is a subtree in the Cayley graph }\Gamma_X\hbox{ containing }1\hbox{ and }g \}$. Let $w$ be the unique freely reduced word representing $g$, and assume $w$ has length $t$. Let $p_w$ denote the path in $\Gamma_X$ from $1$ to $g$ which is labelled by $w$. This path is then the common trunk of all Munn trees $(T,g)$.

Consider the graph obtained from $\Gamma_X$ by contracting $p_w$ into a single vertex $\alpha$.
Notice that this graph is a tree where every vertex except $\alpha$ has valence $2|X|$. 
Consider it as a rooted tree with root $\alpha$. Notice that Munn trees with trunk $p_w$ are in bijection with the finite, rooted subtrees of this infinite rooted tree, which are in turn counted by Fuss-Catalan numbers. All that remains is to determine the valence of $\alpha$ -- this will depend on the length $t$ of the trunk. A vertex in the contracted graph is adjacent to $\alpha$ if its preimage in $\Gamma_X$ was adjacent to a vertex of $p_w$, so it suffices to count how many such vertices there were. 
In $\Gamma_X$, every vertex has degree $2|X|$, so if $w$ is not the empty word, then any internal node of the trunk will connect to exactly $2|X|-2$ vertices of $\Gamma_X$ which do not lie on the trunk, and the endpoints $1$ and $g$ of the trunk connect to $2|X|-1$ vertices not on the trunk. Since there are $t-1 \geq 0$ internal nodes, the number of such vertices is $2(2|X|-1)+(t-1)(2|X|-2)$. If $w$ is the empty word, that is, if $t-1=-1$, then the trunk consists of a single vertex $1$ and it has $2|X|$ neighbours. Thus, $\alpha$ always has valence $2(2|X|-1)+(t-1)(2|X|-2)$.

It follows that $\alpha$ has exactly $2(2|X|-1)+(t-1)(2|X|-2)$ children, while all non-root vertices in the contracted tree have exactly $2|X|-1$ children. Putting $p=2|X|-1$ and $q=2p+(t-1)(p-1)$, this tree is therefore exactly $T_{p,q}$, and the set of all Munn trees with trunk $p_w$ are counted by $R_{p,q}(k)$ exactly. Thus, we have proved:

\begin{lemma}\label{Lem:Munn-tree-breathing}
Let $X$ be a fixed generating set, let $p = 2|X|-1$, and let $t, k \geq 0$. For any fixed trunk of length $t$, the number of Munn trees with that trunk and with $k$ branch edges is exactly $R_{p,q}(k)$, where $q=2p+(t-1)(p-1)$. 
\end{lemma}

Let $M(t,k)$ denote the set of Munn trees with $t$ trunk edges and $k$ branch edges. Taking notation as in Lemma~\ref{Lem:Munn-tree-breathing}, it follows from the same lemma that 
\[
|M(t,k)|=R_{p,q}(k) \cdot|\{w \in (X \cup X^{-1})^* \hbox{ is a freely reduced word of length }t\}|,
\] 
since the second factor in the above expression is the number of possible fixed trunks of length $t$. If $t\geq 1$, then there are $(p+1)p^{t-1}$ possible freely reduced words of length $t$ (we have $p+1$ choices for the first letter, and $p$ for the rest), and  there is one reduced word of length $0$. For any $t \geq 1$, this gives the formula
\begin{align}
|M(t,k)| & = (p+1)p^{t-1}R_{p,2p+(t-1)(p-1)}(k)=\nonumber \\
&=(p+1)p^{t-1}\frac{2p+(t-1)(p-1)}{kp+2p+(t-1)(p-1)} \binom{kp+2p+(t-1)(p-1)}{k},\label{eqn:number-of-trees}
\end{align}
and the particular case of $t=0$ becomes
\begin{equation}
\label{eqn:number-of-idempotents}	
|M(0,k)|=R_{p,p-1}(k) = \frac{p+1}{kp+p+1} \binom{kp+p+1}{k}.
\end{equation}
Thus, we have counted the number of Munn trees $M(t,k)$ of a given number $t$ of trunk edges and $k$ of branch edges.

\subsection{Sphere size}

Using the results of \S\ref{Subsec:Counting-munn-trees}, it is now straightforward to calculate the size $|S(K)|$ of the sphere of radius $K$ in $\FIM(X)$. Any Munn tree in $M(t,k)$ has length $t+2k$ by Lemma~\ref{Lem:Munn-tree-length}. In particular, idempotents (being exactly those Munn trees with trunk length $0$) always have an even length. If $K$ is even, then 
\begin{align}
\begin{split}
\label{eqn:S(K)-K-even}
|S(K)|&=	|M(0,\frac{1}{2}K)|+\sum_{\substack{t+2k = K \\ t \geq 1, \ k \geq 0}}|M(t,k)|=\\
&=\frac{p+1}{\frac{1}{2} Kp+p+1} \binom{\frac{1}{2}Kp+p+1}{\frac{1}{2}K}+\\
& + \sum_{\substack{t+2k = K \\ t \geq 1, \ k \geq 0}}(p+1)p^{t-1}\frac{2p+(t-1)(p-1)}{kp+2p+(t-1)(p-1)} \binom{kp+2p+(t-1)(p-1)}{k},
\end{split}
\end{align}
and when $K$ is odd, we have
\begin{align}
\label{eqn:S(K)-K-odd}
\begin{split}
|S(K)|&=\sum_{\substack{t+2k = K \\ t \geq 1, \ k \geq 0}}|M(t,k)|=\\
&=\sum_{\substack{t+2k= K\\ t \geq 1, \ k \geq 0}}(p+1)p^{t-1}\frac{2p+(t-1)(p-1)}{kp+2p+(t-1)(p-1)} \binom{kp+2p+(t-1)(p-1)}{k}.
\end{split}
\end{align}
By a simple summation, one could also of course deduce corresponding formulae for the size of the ball of a given radius.

\subsection{The monogenic case}\label{Subsec:monogenic}

Before embarking on an in-depth analysis of the asymptotics of \eqref{eqn:S(K)-K-even} and \eqref{eqn:S(K)-K-odd} for $|X|>1$, we consider the much easier monogenic case, i.e.\ when $|X| = 1$. In this case, $p=1$ and \eqref{eqn:number-of-trees} simplifies to give
\[
M(t,k) = 2 \frac{p+1}{kp+p+1} \binom{kp+p+1}{k}= 2 \frac{2}{k+2} \binom{k+2}{k}=2k+2,
\]
for $t \geq 1$ which makes
\[|S(2R+1)|=\sum_{\substack{t+2k= 2R+1 \\ t \geq 1, \ k \geq 0}} 2k+2=\sum_{k=0}^R 2k+2=R^2+3R+2,\]
and, almost identically, we have
\[
|S(2R)|=R+1+\sum_{\substack{t+2k= 2R\\ t \geq 1, \ k \geq 0}} 2k+2=R+1+\sum_{k=0}^{R-1} 2k+2=R^2+2R+1.
\]
Thus, we obtain the following well-known (cf.\ e.g. \cite[Theorem~2.4]{Cutting2001}) result:

\begin{proposition}\label{Prop:FIM1-growth}
The spherical growth rate of the monogenic free inverse monoid $\FIM_1$ is quadratic. Consequently, the ball growth rate of $\FIM_1$ is cubic. 
\end{proposition}

We will now expand in two directions; first, in \S\ref{Sec:Idempotent-growth} we use \eqref{eqn:number-of-idempotents} to find a formula for the growth rate of the number of idempotents in $\FIM(X)$ in Theorem~\ref{Thm:growth-of-idempotents}. Next, in \S\ref{Sec:General-exponential-growth-rate}, we will analyse the formulae \eqref{eqn:S(K)-K-even} and \eqref{eqn:S(K)-K-odd} in depth to find the exponential growth rate of $\FIM(X)$ when $|X|>1$ in Theorem~\ref{expg}.

\section{Idempotent growth}\label{Sec:Idempotent-growth}

\noindent We now calculate the growth of the set of idempotents in $\FIM(X)$ when $|X|>1$. As in \S\ref{Sec:sphere-counting}, we set $p=2|X|-1$, and let $E(X)$ denote the set of idempotents of $\FIM(X)$. 
Recalling that these are represented exactly by the Munn trees with trunk of length $0$, by \eqref{eqn:number-of-idempotents} the number of idempotents in the sphere of radius $K$ is:
\[|S(K) \cap E(X)| = \begin{cases}
 0 \ \ \ \ \ \ \ & \textrm{ if $K$ is odd} \\
 |M(0, k)| = \frac{p+1}{kp+p+1} \binom{kp+p+1}{k} & \textrm{ if $K = 2k$}
 \end{cases}.\]
In fact (as we shall see in the computation that follows) the $K$th roots of the even terms converge, so the exponential growth rate can be computed as a simple limit:
\begin{align}
\lim_{k \to \infty} |M(0,k)|^{\frac{1}{2k}}
&= \underbrace{\left( \lim_{k \to \infty} \left( \frac{p+1}{kp+p+1} \right)^{\frac{1}{2k}} \right)}_{= 1}
\left( \lim_{k \to \infty} {\binom{kp+p+1}{k}}^{\frac{1}{2k}} \right) \nonumber \\
&= \lim_{k \to \infty} {\binom{kp+p+1}{k}}^{\frac{1}{2k}} = \lim_{k \to \infty} \left( \frac{(kp+p+1)!}{      k! \ (k(p-1)+p+1)!  } \right)^{\frac{1}{2k}}.\label{Eq:idem-limit}
\end{align}
Since $|X|>1$ and thus also $p > 1$, we see that both $kp+p+1$ and $k(p-1)+p+1$ tend to infinity with $k$. Hence we may apply Stirling's approximation to \eqref{Eq:idem-limit} to give:
\begin{align*}
&\lim_{k \to \infty} \left( \frac{\sqrt{2\pi (kp+p+1)} \left(\frac{kp+p+1}{e}\right)^{kp+p+1}}{  \sqrt{2\pi k} \left(\frac{k}{e}\right)^{k} \ \sqrt{2\pi (k(p-1)+p+1)}
\left(\frac{k(p-1)+p+1}{e}\right)^{k(p-1)+p+1}   }     \right)^{\frac{1}{2k}} \\
&= \underbrace{\lim_{k \to \infty} \left( \frac{\sqrt{2\pi (kp+p+1)}}{\sqrt{2\pi k} \ \sqrt{2\pi (k(p-1)+p+1)} }     \right)^{\frac{1}{2k}}}_{= 1}
\lim_{k \to \infty} \left( \frac{ \left(\frac{kp+p+1}{e}\right)^{kp+p+1}}{ \left(\frac{k}{e}\right)^{k} \
\left(\frac{k(p-1)+p+1}{e}\right)^{k(p-1)+p+1}   }     \right)^{\frac{1}{2k}}
\end{align*}
We may now cancel the $e$s in the right factor, giving
\begin{align*}
&= \lim_{k \to \infty} \left( \frac{(kp+p+1)^{kp+p+1}}{k^{k} \
(k(p-1)+p+1)^{k(p-1)+p+1}   }     \right)^{\frac{1}{2k}} \\
&=  \left( \lim_{k \to \infty} \frac{(kp+p+1)^{\frac{kp+p+1}{k}}}{k^{\frac{k}{k}} \ (k(p-1)+p+1)^{\frac{k(p-1)+p+1}{k}}} \right)^{\frac{1}{2}} \\
&= \left( \lim_{k \to \infty} \frac{(kp+p+1)^{p}}{k \ (k(p-1)+p+1)^{p-1}} \right)^{\frac{1}{2}}
\underbrace{\left( \lim_{k \to \infty} \frac{(kp+p+1)^{\frac{p+1}{k}}}{(k(p-1)+p+1)^{\frac{p+1}{k}}} \right)^{\frac{1}{2}}}_{=1} \\
&= \left(\lim_{k \to \infty} \frac{(kp+p+1)^p}{k \ (k(p-1)+p+1)^{p-1}}\right)^{\frac{1}{2}} = \sqrt{\frac{p^p}{(p-1)^{p-1}}}
\end{align*}
where the last equality follows by observing that the numerator and denominator are degree $p$ polynomials in $k$, and thus the limit is the quotient of the leading coefficients. Since $\lim_{p \to \infty} \left(\frac{p}{p-1}\right)^{p-1}= e$, it follows that as the rank $|X| \to \infty$, the growth rate of the idempotents tends to $\sqrt{ep}$. We summarize the results below in terms of the rank $r=|X|$, for which $p=2r-1$. 

\begin{theorem}\label{Thm:growth-of-idempotents}
The exponential growth rate of the number of idempotents in the free inverse monoid of rank $r$ is
\[\left(\frac{2r-1}{2r-2}\right)^{r-1} \sqrt{2r-1}.\]
As $r$ tends to infinity this converges to $\sqrt{e(2r-1)}$.
\end{theorem}

For example, the exponential growth rate of idempotents in $\FIM_2$ is $\frac{3}{2} \sqrt{3}$, or approximately $2.598$, while $\sqrt{e(2r-1)} \approx 2.856$. We remark that, in each rank, the exponential growth rate of the idempotents is strictly below $2r-1$, which is the obvious lower bound on the exponential growth rate of the whole monoid. Thus, the proportion of elements of a given length $K$ which are idempotent tends to zero exponentially fast in $K$. In particular, the growth of the idempotents is not a significant contributor to the growth $\FIM(X)$.



\section{The growth of free inverse monoids of higher rank}\label{Sec:General-exponential-growth-rate}

\noindent In this section we describe the exponential growth rate of $\FIM_r$ when $r > 1$ (the case $r=1$ already having been handled in \S\ref{Subsec:monogenic}). This will turn out to be the largest real root of a certain polynomial. We begin with proving some general lemmas regarding roots of certain polynomials.

\subsection{Roots of polynomials}

We shall need the following elementary fact. 
\begin{lemma}\label{sillypoly}
Let $p \geq 2, p\in \mathbb{N}$ and $\alpha, \beta \in \mathbb{R}$ with $\alpha, \beta > 0$. Then, not counting multiplicity, the polynomial $\alpha (t-1)^{p-1} - \beta t^{p-2}\in \mathbb{R}[t]$ has exactly one root in the interval $(1, \infty)$.
\end{lemma}
\begin{proof}
Suppose for contradiction the claim is false and take a counterexample with $p\geq 2$ minimal -- clearly $p >2$, as otherwise the polynomial could not have more than one root. The given polynomial is clearly negative at $t=1$ and positive for sufficiently large $t$, so it must have an odd number of roots (counting multiplicity) on $(1, \infty)$. Because it
is a counterexample to the claim, this means it has at least $3$ distinct roots, or else two distinct roots such that at least one root is also a turning point.  Now between every pair of distinct roots of the polynomial there must be a root of the derivative, so the derivative has at least two distinct roots on $(1, \infty)$. However, the
derivative $\alpha (p-1) (t-1)^{p-2} - \beta (p-2) t^{p-3}$ is another function of the given form but of lower degree, so this contradicts the minimality of $p$.
\end{proof}

Now we use Lemma~\ref{sillypoly} to determine the behaviour of the maximal roots of a certain family of polynomials. These roots will be the exponential growth rates of the free inverse monoids, as we will see in Theorem~\ref{expg}.

\begin{lemma}\label{largerank}
Let $p\in\mathbb{N}$ with $p \geq 2$. Let $y=y(p)$ be the largest real solution to $p^p y^{p-2} - (py-1)^{p-1}=0$. Then $y\in (p,p+1)$ and, as $p \to \infty$, we have
\begin{equation}
y=p+1-\frac{3}{2p}+o\left( \frac{1}{p} \right). \label{Eq:asymptotic-y}
\end{equation}
\end{lemma}

\begin{proof}
We divide the polynomial equation throughout by $p^{2p-2}$ to obtain
\begin{equation*}
\left( \frac{y}{p} \right) ^{p-2} = \left( \frac{y}{p} - \frac{1}{p^2}\right)^{p-1}.
\end{equation*} We will write $z=y/p$. The plan now is to justify that this has a root $z\in (1,1+1/p)$ by using the intermediate value theorem applied to the following polynomial:
\begin{align*}
f(z)=\left( z- \frac{1}{p^2} \right)^{p-1} - z^{p-2}.
\end{align*} 
It is easy to see that $f(1)<0$. Therefore it now suffices for us to show that $f(1+1/p)>0$. This is straightforward to verify when $p=2$, so instead assume $p\geq 3$. We use the binomial expansion to obtain
\begin{align*}
f(z)=z^{p-1}-z^{p-2} - (p-1)\frac{z^{p-2}}{p^2} + \frac{(p-1)(p-2)}{2} \frac{z^{p-3}}{p^4} + \sum_{i=3}^{p-1}(-1)^i{ p-1 \choose i } \frac{z^{p-1-i}}{p^{2i}},
\end{align*}
where the rightmost summation is taken to be $0$ in the case $p=3$. This can be written as
\begin{align*}
f(z)=z^{p-3}\left[ z\left( z-1-\frac{1}{p}+\frac{1}{p^2}\right) +\frac{(p-1)(p-2)}{2p^4} \right] + \sum_{i=3}^{p-1}(-1)^i{ p-1 \choose i } \frac{z^{p-1-i}}{p^{2i}}.
\end{align*} From now on assume $z>1$. By dividing throughout by $z^{p-3}$, we find that
\begin{align} \label{usefulerror}
\frac{f(z)}{z^{p-3}} = \left[ z\left( z-1-\frac{1}{p}+\frac{1}{p^2}\right) +\frac{(p-1)(p-2)}{2p^4} \right] + \sum_{i=3}^{p-1}(-1)^i{ p-1 \choose i } \frac{z^{2-i}}{p^{2i}},
\end{align}
and hence
\begin{align*}
\frac{f(z)}{z^{p-3}} \geq \left[ z\left( z-1-\frac{1}{p}+\frac{1}{p^2}\right) +\frac{(p-1)(p-2)}{2p^4} \right] - \sum_{\substack{i \ \mathrm{odd}\\ i=3}}^{p-1}{ p-1 \choose i } \frac{z^{2-i}}{p^{2i}}.
\end{align*}
Since ${p-1 \choose i}$ is bounded from above by $p^i$ we have
\begin{align*}
\frac{f(z)}{z^{p-3}} \geq \left[ z\left( z-1-\frac{1}{p}+\frac{1}{p^2}\right) +\frac{(p-1)(p-2)}{2p^4} \right] - \sum_{\substack{i \ \mathrm{odd}\\ i=3}}^{p-1}p^i \frac{z^{2-i}}{p^{2i}}.
\end{align*}
Notice that the summation in the final term can be manipulated to obtain
\begin{align*}
- \sum_{\substack{i \ \mathrm{odd}\\ i=3}}^{p-1}p^i \frac{z^{2-i}}{p^{2i}}>- (z^{-1}p^{-3}+z^{-3}p^{-5}+z^{-5}p^{-7}+ \cdots ) = - \frac{1}{zp^3 \left(1-\frac{1}{z^2p^2}\right)}.
\end{align*} 
Hence for all $z>1$
\begin{align*}
\frac{f(z)}{z^{p-3}} >\left[ z\left( z-1-\frac{1}{p}+\frac{1}{p^2}\right) +\frac{(p-1)(p-2)}{2p^4} \right] - \frac{1}{zp^3 \left(1-\frac{1}{z^2p^2}\right)}.
\end{align*} Now substituting in $z=1+\frac{1}{p}$, we get
\begin{align*}
\frac{f(1+\frac{1}{p})}{(1+\frac{1}{p})^{p-3}}>\left[ \frac{3}{2p^2}-\frac{1}{2p^3} +\frac{1}{p^4} \right] - {\frac{1}{zp^3 \left(1-\frac{1}{z^2p^2}\right)}},
\end{align*}
where the rightmost denominator is
\begin{align*}
zp^3 \left(1-\frac{1}{z^2p^2}\right)=(p^3+p^2)\left(1-\frac{1}{p^2+2p+1}\right) > p^3,
\end{align*}
from which it follows that
\begin{align*}
\frac{f(1+\frac{1}{p})}{(1+\frac{1}{p})^{p-3}} > \left[ \frac{3}{2p^2}-\frac{1}{2p^3} +\frac{1}{p^4} \right]-\frac{1}{p^3} > \frac{1}{p^4}>0.
\end{align*}

Thus, we have proved that $f(1+\frac{1}{p})>0$. Since $f(1)<0$, by the intermediate value theorem there is a root $z\in(1,1+1/p)$ and as $z = y/p$, there is a root $y\in(p,p+1)$ as claimed. Hence, by Lemma~\ref{sillypoly} (applied with $t=py$) the root $y$ is the maximal root of our polynomial because it is the unique root in the interval $(\frac{1}{p},\infty)$.

We now prove the asymptotic formula \eqref{Eq:asymptotic-y} as $p\rightarrow\infty$. Let $z$ be the unique root of $f(z)$ in the interval $(1, 1+\frac{1}{p})$, and set $z=1+1/p+C/p^2$ -- note that $C$ depends on $p$, and $-p < C < 0$. Substitute this formula for $z$ into \eqref{usefulerror}. Elementary algebraic manipulations then yield
\begin{align*}
0=\frac{f(z)}{z^{p-3}}=\left[ \frac{2C+3}{2p^2} +\frac{2C-1}{2p^3}+\frac{C^2 +C +1}{p^4} \right] + \sum_{i=3}^{p-1}(-1)^i{ p-1 \choose i } \frac{z^{2-i}}{p^{2i}}.
\end{align*}
Note that since $z >1$ by assumption, we have $z^{2-i} < 1$ whenever $i \geq 3$, and thus
\begin{align*}
\left|\sum_{i=3}^{p-1}(-1)^i{ p-1 \choose i } \frac{z^{2-i}}{p^{2i}}\right|&\leq \sum_{i=3}^{p-1}{ p-1 \choose i } \frac{z^{2-i}}{p^{2i}} \leq \sum_{i=3}^{p-1}{ p-1 \choose i } \frac{1}{p^{2i}}
\\ &\leq \sum_{i=3}^{p-1} \frac{1}{p^{i}} \leq \frac{1}{p^3} + \frac{p-5}{p^4} < \frac{2}{p^3}.
\end{align*}
Furthermore, since $|C|\leq p$, we have that
\begin{align*}\left|\frac{2C-1}{2p^3}+\frac{C^2 +C +1}{p^4}\right|\leq \frac{2|C|+|1|}{2p^3}+\frac{|C^2| +|C| +|1|}{p^4} < \frac{5}{p^2}.
\end{align*}
Therefore as $p\rightarrow\infty$, we get
\begin{align*}
0= \frac{f(z)}{z^{p-3}} =  \frac{2C+3}{2p^2} +O\left( \frac{1}{p^2}\right),
\end{align*}
In particular,
$0=2C+3+O(1),$ that is, $C=-3/2+O(1)$ and therefore $C$ is bounded.
This allows us to refine the estimations above and observe that $\frac{2C-1}{2p^3}+\frac{C^2 +C +1}{p^4}$ is also $o(1/p^2)$, hence
\begin{align*}
	0= \frac{f(z)}{z^{p-3}} =  \frac{2C+3}{2p^2} +o\left( \frac{1}{p^2}\right),
\end{align*}
which in turn yields $C=-3/2+o(1)$.
 Substituting this back into $z=1+1/p+C/p^2$, we obtain that the  unique root for the polynomial in $z=y/p$ is $1+1/p-3/(2p^2)+o(1/p^2)$. This proves \eqref{Eq:asymptotic-y}.
\end{proof}

\subsection{Exponential growth of free inverse monoids}

We will now use the sizes of the spheres determined in  \S\ref{Sec:sphere-counting} to determine the exponential growth rate of $\FIM_r$ when $r>1$. 

Recall that the growth rate is defined as
\[
\limsup_{K\rightarrow\infty} \left( |S(K)| \right)^{\frac{1}{K}}.
\]
There are only polynomially many terms in our formulae \eqref{eqn:S(K)-K-even} and \eqref{eqn:S(K)-K-odd} for $|S(K)|$ with respect to $K$, so the exponential growth rate $y$ is determined by the growth rate of the sequence of maximal terms. Moreover, by Theorem~\ref{Thm:growth-of-idempotents}, the exponential growth rate of the idempotents is strictly less than $2 |X| -1$, which is a lower bound on the exponential growth rate of $\FIM(X)$, so the term (which occurs in the even case only) corresponding to trunk length $0$ may be ignored. Thus the growth rate is equal to
\begin{equation}
\label{eqn:growth}
\limsup_{K\rightarrow \infty} \max_{\substack{t+2k= K\\ t \geq 1, k \geq 0}} |M(t, k)|^{\frac{1}{K}}.
\end{equation}
Clearly we can choose a sequence of values $t_i, k_i$ and $K_i$ for $t, k$ and $K$ respectively such that the limit $\lim_{i \to \infty}|M(t_i, k_i)|^{\frac{1}{K_i}}$ of the corresponding terms converges to the growth rate. We begin with a lemma which will help us identify key properties of the sequences $t_i, k_i$. 

\begin{lemma}\label{lemma_yvalues}
Let $K_i$, $t_i$ and $k_i$ ($i \in \N$) be non-negative integer sequences such that $t_i \geq 1$,  $K_i = t_i + 2k_i$, and $K_i \to \infty$ as $i \to \infty$. Assume furthermore that
\[
y = \lim_{i \to \infty} M(t_i, k_i)^{\frac{1}{K_i}}
\]
is defined and finite,
where $M(t_i, k_i)$ denotes the set of Munn trees over the alphabet $X$ with $t_i$ trunk edges and $k_i$ branch edges as in \S3. If as $i \to \infty$ either $k_i$ is bounded above or $\frac{t_i-1}{k_i}$ is not bounded above, then $y \leq p$. Otherwise, if $x$ is any accumulation point of the sequence $\frac{t_i-1}{k_i}$, then
\begin{multline}\label{Eq:y-expression-from-lemma}
y = \exp \frac{1}{x+2} \biggl( x\log (p)+ \left( p+ x(p-1) \right) \log \left( p+ x(p-1) \right) \\ - \left( (p-1) + x(p-1) \right) \log \left( (p-1) + x(p-1) \right)
\biggr).
\end{multline}

\end{lemma}
\begin{proof}
By \eqref{eqn:number-of-trees} we have 
\begin{align}
y & =\lim_{i\rightarrow \infty} \left(  (p+1)p^{t_i-1}\frac{2p+(t_i-1)(p-1)}{k_ip+2p+(t_i-1)(p-1)} \binom{k_ip+2p+(t_i-1)(p-1)}{k_i}\right)^{ \frac{1}{K_i}} \nonumber \\
& =\lim_{i\rightarrow \infty}\underbrace{ \left(  (p+1)\frac{2p+(t_i-1)(p-1)}{k_ip+2p+(t_i-1)(p-1)} \right)^{ \frac{1}{K_i}}}_{\to 1 \text{ as } i \to \infty} \left( p^{t_i-1} \binom{k_ip+2p+(t_i-1)(p-1)}{k_i}\right)^{ \frac{1}{K_i}} \nonumber \\
& =\lim_{i\rightarrow \infty} \left( p^{t_i-1} \binom{k_ip+2p+(t_i-1)(p-1)}{k_i}\right)^{ \frac{1}{K_i}}. \label{Eq:binomial-limit}
\end{align}

Suppose first that the sequence $k_i$ is bounded above by some constant $B$. Then, since $\binom{n}{r} \leq n^r$, we can bound the last term in \eqref{Eq:binomial-limit}, yielding
\begin{align*}
y \leq \lim_{i\rightarrow \infty} \left( p^{t_i-1} \right)^{\frac{1}{K_i}} \underbrace{\left( (Bp+2p+(t_i-1)(p-1))^B \right)^{\frac{1}{K_i}}}_{\to 1 \text{ as $i \to \infty$}} = \lim_{i\rightarrow \infty} \left( p^{t_i-1} \right)^{\frac{1}{K_i}} \leq p,
\end{align*}
where the final inequality is due to the fact that $t_i \leq K_i$ for all $i$. Thus, the claim is proved in the case where $k_i$ is uniformly bounded above, so assume that it is not.

In this case, by passing to a subsequence if necessary, we may assume that $k_i \to \infty$ as $i \to \infty$. Hence the three sequences $k_ip+2p+(t_i-1)(p-1)$, $k_i$, and $k_i(p-1)+2p+(t_i-1)(p-1)$, all tend to infinity as $i \to \infty$. Thus, we may use Stirling's approximation for each of these three sequences to replace 
\[
\binom{k_ip+2p+(t_i-1)(p-1)}{k_i}
\]
in \eqref{Eq:binomial-limit} by the expression
\[
\frac{ (k_ip+2p+(t_i-1)(p-1))^{k_ip+2p+(t_i-1)(p-1)}}{k_i^{k_i}{(k_i(p-1)+2p+(t_i-1)(p-1))^{k_i(p-1)+2p+(t_i-1)(p-1)}}} \alpha(t_i,k_i)
\]
where $\alpha$ is a rational function of the $t_i$ and $k_i$ such that $\alpha(t_i, k_i)^{\frac{1}{K_i}} \to 1$ as $i \to \infty$. Doing this substitution into \eqref{Eq:binomial-limit} thus eliminates the $\alpha(t_i, k_i)$ in the limit, and we are left to analyse
\[
y =\lim_{i\rightarrow \infty} \left( p^{t_i-1} \frac{ (k_ip+2p+(t_i-1)(p-1))^{k_ip+2p+(t_i-1)(p-1)}}{k_i^{k_i}{(k_i(p-1)+2p+(t_i-1)(p-1))^{k_i(p-1)+2p+(t_i-1)(p-1)}}}\right)^{ \frac{1}{K_i}}.
\]
From this point on, for ease of calculation, we set $n_i = t_i-1$ so that $K_i = t_i + 2 k_i = n_i + 1 + 2 k_i$ and take logarithms, yielding
\begin{align}\label{Eq:y-initial-expression}
y = \lim_{i\rightarrow\infty} \exp \frac{1}{n_i+1+2k_i}\left( n_i \log(p) +\log Q(p, n_i, k_i) \right),
\end{align}
where we have set
\[
Q(p, n_i, k_i) = \frac{ (k_ip+2p+n_i(p-1))^{k_ip+2p+n_i(p-1)}}{k_i^{k_i}{(k_i(p-1)+2p+n_i(p-1))^{k_i(p-1)+2p+n_i(p-1)}}}.
\]
Dividing the numerator and denominator of $Q(p, n_i,k_i)$ by $k_i^{k_ip + 2p + n_i(p-1)}$, then factoring out $k_i$ factors from all exponents, we obtain 
\begin{equation}\label{Eq:logQ-expr-1}
\log Q(p, n_i, k_i) = k_i \log \frac{\left(p + \frac{2p}{k_i} + \frac{n_i}{k_i}(p-1)\right)^{p + \frac{2p}{k_i} + \frac{n_i}{k_i}(p-1)}}{\left( p-1 + \frac{2p}{k_i} + \frac{n_i}{k_i}(p-1)\right)^{p-1 + \frac{2p}{k_i} + \frac{n_i}{k_i}(p-1)}}.
\end{equation}
By using the properties of logarithmic functions, we can rewrite \eqref{Eq:logQ-expr-1} as 
\begin{align}\label{Eq:logQ-expr-2}
\log Q(p, n_i, k_i) &=  k_i \left[ \left( p + \frac{2p}{k_i} + \frac{n_i}{k_i} (p-1) \right) \log \left( p + \frac{2p}{k_i} + \frac{n_i}{k_i} (p-1) \right) \nonumber \right. \\
&\left. - \left( p-1 + \frac{2p}{k_i} + \frac{n_i}{k_i}(p-1) \right) \log\left(p-1 + \frac{2p}{k_i} + \frac{n_i}{k_i}(p-1)\right)\right].
\end{align}
Now, using the simple equality
\[
\frac{k_i}{n_i + 1 + 2k_i} = \frac{1}{\frac{n_i}{k_i} + \frac{1}{k_i} + 2}
\]
together with \eqref{Eq:logQ-expr-2}, we can rewrite the expression \eqref{Eq:y-initial-expression} into 
\begin{align} \label{beforexlimit}
y = \lim_{i \rightarrow \infty} &\exp \frac{1}{\frac{n_i}{k_i}+\frac{1}{k_i}+2} \Biggl[ \frac{n_i}{k_i}\log(p) + \nonumber \\ 
&+\overbrace{\left(p+\frac{2p}{k_i}+\frac{n_i}{k_i}(p-1)\right) \log \left( p+\frac{2p}{k_i} +\frac{n_i}{k_i}(p-1) \right)}^{=: \: \beta(p, n_i, k_i)} \nonumber 
\\ &- \underbrace{\left( (p-1) + \frac{2p}{k_i} + \frac{n_i}{k_i}(p-1) \right) \log \left( (p-1) + \frac{2p}{k_i} + \frac{n_i}{k_i}(p-1) \right)}_{=:\: \gamma(p, n_i, k_i)}
\biggr].
\end{align} 
Suppose first that $\frac{n_i}{k_i}$ is not uniformly bounded above. Then we may refine our subsequence so that $\frac{n_i}{k_i} \to \infty$ as $i \to \infty$. We consider the behaviour of the quantity $\beta(p, n_i, k_i) - \gamma(p, n_i, k_i)$ appearing in \eqref{beforexlimit}. It can be rewritten as
\[
\log \left( p + \frac{2p}{k_i} + \frac{n_i}{k_i}(p-1) \right) + \underbrace{\left( p-1 + \frac{2p}{k_i} + \frac{n_i}{k_i}(p-1) \right) \log \left( \frac{p+\frac{2p}{k_i} + \frac{n_i}{k_i}(p-1)}{p-1 + \frac{2p}{k_i} + \frac{n_i}{k_i}(p-1)}\right)}_{\to 1 \text{ as $i \to \infty$}}
\]
and hence as $i \to \infty$, we have that $\beta(p, n_i, k_i) - \gamma(p, n_i, k_i)$ tends to
\[
\log \left( p + \frac{2p}{k_i} + \frac{n_i}{k_i}(p-1) \right) + 1.
\]
In this way, the limit \eqref{beforexlimit} becomes
\begin{align*}
y &= \lim_{i \to \infty} \exp \frac{1}{\frac{n_i}{k_i} + \frac{1}{k_i} + 2}\left( \frac{n_i}{k_i} \log(p) + \log\left( p + \frac{2p}{k_i} + \frac{n_i}{k_i} (p-1) \right) + 1 \right) \\
&= \lim_{i \to \infty} \exp \left( \underbrace{\frac{\frac{n_i}{k_i}}{\frac{n_i}{k_i} + \frac{1}{k_i} + 2}}_{\to 1 \text{ as $i \to \infty$}} \log(p) + \underbrace{\frac{\log\left( p + \frac{2p}{k_i} + \frac{n_i}{k_i} (p-1) \right) + 1}{\frac{n_i}{k_i} + \frac{1}{k_i} + 2}}_{\to 0 \text{ as $i \to \infty$}} \right) = p.
\end{align*}
This is precisely what was to be shown. 

There remains only the case where $\frac{n_i}{k_i}$ is bounded above. In this case by refining the sequence we may assume it converges to some $x\in [0,\infty)$. In this case the expression \eqref{beforexlimit} can, by taking the limit and recalling that $k_i\rightarrow \infty$, be seen to be changed into precisely 
\begin{multline*}
	y = \exp \frac{1}{x+2} \biggl( x\log (p)+ \left( p+ x(p-1) \right) \log \left( p+ x(p-1) \right) \\ - \left( (p-1) + x(p-1) \right) \log \left( (p-1) + x(p-1) \right)
	\biggr)
\end{multline*}
as required. 
\end{proof}

Having analyzed the limit in Lemma~\ref{lemma_yvalues}, we set
\begin{align}
\label{eqn:h(x)}
		h(x)=\frac{1}{x+2} \biggl( x\log (p)+ \nonumber
		\left( p+ x(p-1) \right) \log \left( p+ x(p-1) \right) \\ -\left( (p-1) + x(p-1) \right) \log \left( (p-1) + x(p-1) \right)
		\biggr).
	\end{align}
With this notation, an easy consequence of Lemma \ref{lemma_yvalues} is now the following key lemma.

\begin{lemma}
\label{lemma:growth_rate_is_max}
The exponential growth rate of $\FIM(X)$ is equal to 
\[
\max\left\{p, \sup_{x \geq 0} \exp h(x)\right\}.
\]
\end{lemma}

\begin{proof}
It follows immediately by Lemma \ref{lemma_yvalues} that the growth rate is at most the maximum above. For the converse inequality, notice that for any sequence $K_i$ with $i \to \infty$, there exist sequences $t_i, k_i$ with $K_i=t_i+2k_i$ and satisfying either of the three properties described in the conditions of Lemma~\ref{lemma_yvalues}: there are clearly such sequences where $k_i$ is bounded above, and also sequences where $\frac{t_i-1}{k_i}$ is not bounded above. Finally, for any given $x \in [0, \infty)$, we can choose $t_i, k_i$ such that $\frac{t_i-1}{k_i}$ converges to $x$: indeed, it suffices to set each $k_i$ to be the integer part of $\frac{K_i}{x+2}$ and then define $t_i = K_i - 2k_i$. For any of the sequences above, by passing to a subsequence if necessary, we can also ensure that the limit $\lim_{i \to \infty} M(t_i, k_i)^{\frac{1}{K_i}}$ exists. It follows from  Lemma~\ref{lemma_yvalues} that we can realize any number in $\{p\} \cup \{\exp h(x): x \geq 0\}$ as a limit of a sequence $M(t_i, k_i)^{K_i}$, and clearly any such number must be a lower bound for the growth rate \eqref{eqn:growth}.
\end{proof}

We are ready to state and prove our second main theorem:

\begin{theorem}\label{expg}
Let $\FIM(X)$ be the free inverse monoid on $X$. Assume $|X|\geq 2$, and write $y$ for the exponential growth rate of $\FIM(X)$. Write $p=2|X|-1$. Then $y\in (p,p+1)$ is the largest real root of the polynomial equation
\begin{equation*}
p^p y^{p-2}-(py-1)^{p-1} = 0.
\end{equation*} In particular, $y$ is an algebraic number.
\end{theorem}

\begin{proof}
The fact that the maximal root of the polynomial above lies inside $(p,p+1)$ is stated and proved in Lemma~\ref{largerank} so what remains to be shown is that
the exponential growth rate is a root, and is greater than $p$. By Lemma \ref{lemma:growth_rate_is_max}, we have
\[
y=\max\left\{p, \sup_{x \geq 0} \exp h(x)\right\}.
\]
It is easy to see that as $x\rightarrow \infty$, $\exp h(x)$ tends to $p$, while at $x=0$ the function is equal to the growth rate of the idempotents, which is less than $p$, hence it suffices to analyze the local extrema of $\exp h(x)$ and compare these values with $p$. Since the exponential function is monotonically increasing, the local extrema of $\exp h(x)$ and $h(x)$ occur at the same places, and we can directly determine the latter by computing the derivative of $h(x)$ and setting it to $0$. Multiplying the equation $h'(x)=0$ by $(x+2)^2$, cancelling terms and rearranging, we obtain
\begin{align}\label{Eq:h-derivative}
2\log p + (p-2) \log(p+x(p-1)) = (p-1) \log((p-1)+x(p-1)).
\end{align} 
Multiplying both sides by $(x+1)$, we get
\[
2(x+1)\log p + (p-2)(x+1) \log(p+x(p-1)) = (x+1)(p-1) \log((p-1)+x(p-1))
\]
and we find the right-hand side of the above expression in \eqref{eqn:h(x)}. Replacing it therein by the left-hand side, we find that when $h'(x)=0$, the value of $\exp h(x)$ is 
\begin{align*}
y = \exp \frac{1}{x+2} &[ x\log p +(p+x(p-1))\log(p+x(p-1)) \\ &\qquad- 2(x+1)\log p - (p-2)(x+1) \log(p+x(p-1))] = 
\\
&= \exp \frac{1}{x+2} [(-x-2)\log p + (2+x) \log(p+x(p-1))] \\ &=\frac{p+x(p-1)}{p} > p.
\end{align*}
Thus, when $x$ maximizes $\exp h(x)$, we have that $x = \frac{p}{p-1}(y-1)$. On the other hand, by exponentiating both sides of \eqref{Eq:h-derivative}, we see that $h(x)$ is maximized when $x$ is a solution to the polynomial equation
\begin{align}\label{polyforx} 
p^2(p+x(p-1))^{p-2} = ((p-1)(x+1))^{p-1}.
\end{align} 
Substituting $x = \frac{p}{p-1}(y-1)$ into \eqref{polyforx}, we find that the maximal value of $\exp h(x)$, i.e.\ the exponential growth rate of $\FIM(X)$, satisfies
\begin{align*}
p^2(p+p(y-1))^{p-2} = [ p(y-1) + (p-1) ] ^{p-1},
\end{align*} simplifying to $p^py^{p-2}=(py-1)^{p-1}$, as required. 
\end{proof}

Combining the estimates in Lemma~\ref{largerank} with Theorem~\ref{expg}, we have the following corollary on the behaviour of the growth rate for large ranks.

\begin{corollary}\label{corlargerank}
Let $\FIM(X)$ be the free inverse monoid on $X$ and write $y=y(|X|)$ for the exponential growth rate of $\FIM(X)$. Then
\begin{equation*}
y=2|X|-\frac{3}{4 |X|} +o\Bigl( \frac{1}{|X|}\Bigr).
\end{equation*} In particular, the exponential growth rate tends to twice the rank as $|X|\rightarrow \infty$.
\end{corollary}

\begin{proof} It is straightforward to check that the corollary follows from Lemma~\ref{largerank} and Theorem~\ref{expg} as
\begin{equation*}
\frac{1}{p}=\frac{1}{2|X|-1}=\frac{1}{2|X|} + o\left( \frac{1}{|X|} \right),
\end{equation*} and anything $o(1/p)$ is automatically $o(1/|X|)$. \end{proof}

\begin{remark}\label{Rem:remark-irr}
The polynomial \eqref{Eq:main-polynomial} seems almost always to be irreducible over $\mathbb{Q}$. A curious exception is the rank $5$ ($p=9$) case, where the polynomial factorizes into irreducibles as
\[
-\left( (py)^2 + py + 1 \right) \left( (py)^6 - 90 (py)^5+ 117(py)^4 - 83(py)^3 + 36(py)^2 - 9py + 1 \right).
\]
We have verified
computationally that the polynomial is irreducible for all other ranks up to $2000$. It would be interesting to know if there is a conceptual reason why the polynomial is not irreducible for $\FIM_5$, and indeed whether the polynomial is irreducible in all other cases.
\end{remark}

\begin{example}
Consider $r = 2$, i.e.\ we wish to determine the exponential growth rate of $\FIM_2$, which should lie between $3$ and $4$. By Theorem~\ref{expg} this growth rate is the largest real root of the polynomial equation $27y - (3y - 1)^2 = 0$. Hence the exponential growth rate of $\FIM_2$ is exactly $\frac{11}{6} + \frac{\sqrt{13}}{2}$, or approximately $3.636$.
\end{example}

Finally, we pose the following question, which is a first step in expanding the work in this article into a program of extracting algebraic information purely from growth in the setting of inverse monoids: 

\begin{question}
Let $M$ be an inverse monoid generated by $r$ elements such that $M$ has the same growth rate as $\FIM_r$. Is $M$ necessarily free? 
\end{question}

The corresponding question for free groups has an affirmative answer \cite{Koubi1998}, but the proof uses the Nielsen--Schreier Theorem, the analogue of which fails for free inverse monoids.

\section*{Acknowledgements}
The authors thank Tara Brough and Marianne Johnson for helpful comments and discussions.

\bibliographystyle{amsalpha}
\bibliography{free-inverse-arxiv-July2024.bib}

\end{document}